\documentclass[12pt,reqno]{amsart}
\usepackage{amssymb,amsmath,amsthm,hyperref}
\newtheorem{theorem}{Theorem}

\newtheorem{corollary}[theorem]{Corollary}
\newtheorem{proposition}[theorem]{Proposition}

\theoremstyle{remark}
\newtheorem*{remark}{Remark}
\newtheorem*{claim}{Claim}
\theoremstyle{definition}

\numberwithin{theorem}{section} \numberwithin{equation}{section}
\numberwithin{example}{section}

\title[Ramanujan's radial limits]{Ramanujan's radial limits and mixed mock modular bilateral $q$-hypergeometric series}

\author{Eric Mortenson}

\begin{document}

\date{3 June 2014}



\begin{abstract}
Using results from Ramanujan's lost notebook, Zudilin recently gave an insightful proof of a radial limit result of Folsom, Ono, and Rhoades for mock theta functions.  Here we see that the author's previous work on the dual nature of Appell--Lerch sums and partial theta functions and on constructing bilateral $q$-series with mixed mock modular behaviour is well-suited for such radial limits.  We present five more radial limit results which follow from mixed mock modular bilateral $q$-hypergeometric series.  We also obtain the mixed mock modular bilateral series for a universal mock theta function of Gordon and McIntosh.  The later bilateral series can be used to compute radial limits for many classical second, sixth, eight, and tenth order mock theta functions.
\end{abstract}


\address{Max-Planck-Institut f\"ur Mathematik, Vivatsgasse 7, 53111 Bonn, Germany}
\email{etmortenson@gmail.com}
\maketitle
\setcounter{section}{-1}

\section{Introduction}

In his last letter to Hardy, Ramanujan gave a list of seventeen functions which he called ``mock theta functions.''  Ramanujan defined a mock theta function $f(q)$ as a $q$-series, convergent  for $|q|<1$, such that for every root of unity $\zeta$, there is a theta function $\theta_{\zeta}(q)$ such that the difference $f(q)-\theta_{\zeta}(q)$ is bounded as $q\rightarrow \zeta$ radially; moreover, there is no single theta function which works for all $\zeta$.   Only recently have mock theta functions been shown to satisfy Ramanujan's mysterious definition \cite{GOR}.

Ramanujan also made slightly more specific claims about asymptotic behaviour. We give an example.
Referring the reader unfamiliar with standard notation to Section \ref{section:notation}, we recall the third order mock theta function
\begin{align}
f(q):=\sum_{n=0}^{\infty}\frac{q^{n^2}}{(-q;q)_n^2}=1+\sum_{n=1}\frac{q^{n^2}}{(1+q)^2(1+q^2)^2\cdots(1+q^n)^2}\label{equation:f-def}
\end{align}
and the theta function 
\begin{align}
b(q):&=(1-q)(1-q^3)(1-q^5)\cdots \times(1-2q+2q^4-2q^9+\cdots)\label{equation:b-def}\\
&=(q;q^2)_{\infty}\sum_{n=-\infty}^{\infty}(-1)^nq^{n^2}=J_{1,2}^2/J_1\notag
\end{align}

\begin{claim}[Ramanujan] As $q$ approaches an even root of unity of order $2k$, the difference $f(q)-(-1)^kb(q)$ is absolutely bounded.
\end{claim}
Recently, Folsom, Ono, and Rhoades not only proved but also refined Ramanujan's claim:
\begin{theorem}\cite{FOR1,FOR2} \label{theorem:thm-for} If $\zeta$ is a primitive even order $2k$ root of unity, then, as $q$ approaches $\zeta$ radially within the unit disk, we have that
\begin{equation*}
\lim_{q\rightarrow \zeta} \Big ( f(q)-(-1)^kb(q)\Big )=-4\cdot \sum_{n=0}^{k-1}(1+\zeta)^2(1+\zeta^2)^2\cdots(1+\zeta^n)^2\zeta^{n+1}.
\end{equation*}
\end{theorem}
\noindent Folsom, Ono, and Rhoades also proved a more general result of Theorem \ref{theorem:thm-for}, see \cite[Theorem $1.2$]{FOR1, FOR2}.

Zudilin \cite{Zu} gave an insightful proof of Theorem \ref{theorem:thm-for} using ingredients from Ramanujan's Lost Notebook \cite{RLN}.  Upon studying Zudilin's proof, we see that the ingredients for six radial limit examples are contained in the author's work on the dual nature of Appell--Lerch sums and partial theta functions with applications to constructing mixed mock modular bilateral $q$-hypergeometric series  \cite{Mo1}.   

We note that the $q$-hypergeometric series in (\ref{equation:f-def}) is a unilateral series in the sense that the lower summation limit is finite but the upper limit is infinite.  If both the upper or lower limits are infinite, then we have a bilateral series, for example
\begin{equation}
\sum_{n=-\infty}^{\infty}\frac{q^{n^2}}{(-q;q)_n^2}=\sum_{n=0}^{\infty}\frac{q^{n^2}}{(-q;q)_n^2}+\sum_{n=-\infty}^{-1}\frac{q^{n^2}}{(-q;q)_n^2},\label{equation:bilat-ex}
\end{equation}
where sum from $-\infty$ to $-1$ on the right-hand side of (\ref{equation:bilat-ex}) is called the {\em tail} of the bilateral series found on the left-hand side (\ref{equation:bilat-ex}).

One such radial limit which can be obtain from \cite{Mo1} reads
\begin{theorem}\label{theorem:radial-new} If $\zeta$ is a primitive odd order $2k+1$ root of unity, then, as $q$ approaches $\zeta$ radially within the unit disk, we have that
\begin{equation*}
\lim_{q\rightarrow \zeta} \Big ( q w(q)-\theta(q) \Big )=-\sum_{n=0}^{k}(1-\zeta)^2(1-\zeta^3)^2\cdots(1-\zeta^{2n-1})^2\zeta^{2n+1},\label{equation:new-rad}
\end{equation*}
where 
\begin{equation*}
w(q):=\sum_{n=0}^{\infty}\frac{q^{2n(n+1)}}{(q;q^2)_{n+1}^2} {\text{ and }} \theta(q):=\frac{\overline{J}_{1,4}^2\overline{J}_{2,4}}{\overline{J}_{0,4}J_{2}},
\end{equation*}
are a third order mock theta function and a theta function, respectively.
\end{theorem}
In this note we use the setting of mixed mock modular bilateral $q$-hypergeometric series to give five more examples of radial limits of mock theta functions.  In Section \ref{section:notation} we recall basic notation and useful results on Appell--Lerch sums.  In Section \ref{section:zudilin} we discuss Zudilin's proof so that the reader can see the various pieces of all six examples. In Section \ref{section:new-theorem} we prove Theorem \ref{theorem:radial-new}, and in Section \ref{section:radials} we give four more radial limit results.  In Section \ref{section:GM} we derive the bilateral series with mixed mock modular behaviour for the universal mock theta function $g_2(x,q)$ of Gordon and McIntosh \cite{GM}.    Using identities found in \cite{GM} one can then obtain radial limits for second, sixth, eighth, and tenth order mock functions which can be written in terms of $g_2(x,q)$.

In light of \cite{Zu}, we see that the author's work \cite{Mo1} is surprisingly well-suited for radial limit examples such as \cite{BKLR, FOR1, FOR2} and for addressing questions found at the end of \cite{R}.   We point out that \cite{BKLR} considers radial limits of known mock theta function identities that can be written as bilateral series with modular behaviour.  Our bilateral series are, to the best of our knowledge, new; moreover, they have mixed mock modular behaviour. 

\section{Notation and Technical Details}\label{section:notation}

 Let $q$ be a nonzero complex number with $|q|<1$ and define $\mathbb{C}^*:=\mathbb{C}-\{0\}$.  Recall
 \begin{allowdisplaybreaks}
\begin{gather}
(x)_n=(x;q)_n:=\prod_{i=0}^{n-1}(1-q^ix), \ \ (x)_{\infty}=(x;q)_{\infty}:=\prod_{i\ge 0}(1-q^ix),\notag \\
{\text{and }} \ \ j(x;q):=(x)_{\infty}(q/x)_{\infty}(q)_{\infty}=\sum_{n=-\infty}^{\infty}(-1)^nq^{\binom{n}{2}}x^n,\label{equation:theta-def}
\end{gather}
\end{allowdisplaybreaks}%
where in the last line the equivalence of product and sum follows from Jacobi's triple product identity.    Here $a$ and $m$ are integers with $m$ positive.  Define
\begin{gather*}
J_{a,m}:=j(q^a;q^m), \ \ J_m:=J_{m,3m}=\prod_{i\ge 1}(1-q^{mi}), \ {\text{and }}\overline{J}_{a,m}:=j(-q^a;q^m).
\end{gather*}
We define our building block for Hecke-type double sums as in \cite{HM}:
\begin{equation}
f_{a,b,c}(x,y,q):=\Big (\sum_{r,s\ge0} -\sum_{r,s<0}\Big ) (-1)^{r+s}x^ry^sq^{a\binom{r}{2}+brs+c\binom{s}{2}}.\label{equation:fabc-def}\\
\end{equation}
We will also use the following definition of an Appell--Lerch sum \cite{HM}:
\begin{equation}
m(x,q,z):=\frac{1}{j(z;q)}\sum_{r=-\infty}^{\infty}\frac{(-1)^rq^{\binom{r}{2}}z^r}{1-q^{r-1}xz}.\label{equation:mdef-eq}
\end{equation}
The Appell--Lerch sum $m(x,q,z)$ satisfies several functional equations and identities, which we collect in the form of a proposition:

 \begin{proposition}  \cite{L1, HM} For generic $x,z,z_0,z_1\in \mathbb{C}^*$
{\allowdisplaybreaks \begin{subequations}
\begin{equation}
m(x,q,z)=m(x,q,qz),\label{equation:m-fnq-z}
\end{equation}
\begin{equation}
m(x,q,z)=x^{-1}m(x^{-1},q,z^{-1}),\label{equation:m-fnq-flip}
\end{equation}
\begin{equation}
m(qx,q,z)=1-xm(x,q,z),\label{equation:m-fnq-x}
\end{equation}
\begin{equation}
m(x,q,z_1)-m(x,q,z_0)=\frac{z_0J_1^3j(z_1/z_0;q)j(xz_0z_1;q)}{j(z_0;q)j(z_1;q)j(xz_0;q)j(xz_1;q)},\label{equation:m-change-z}
\end{equation}
\begin{equation}
m(x,q,z)=m(x,q,x^{-1}z^{-1}).\label{equation:m-fnq-zflip}
\end{equation}
\end{subequations}}
\end{proposition}



The next proposition is based on equations of \cite{RLN} many of which were proved in \cite{AM}.  The symbol $\sum^{\ast}$ indicates convergence problems, so care should be taken.

\begin{proposition}\cite[Proposition 2.6]{Mo1}\label{proposition:eulerian-mxqz-prop} We have
{\allowdisplaybreaks \begin{align}
(1+x^{-1})\sum_{n= 0}^{\infty}\frac{q^{n+1}(-q)_{2n}}{(qx,q/x;q^2)_{n+1}}&=-m(x,q^2,q)\label{equation:RLNid1}\\
\sum_{n= 0}^{\infty}\frac{(-1)^nq^{n^2}(q;q^2)_{n}}{(-x;q^2)_{n+1}(-q^2/x;q^2)_{n}}&=m(x,q,-1)+\frac{J_{1,2}^2}{2j(-x;q)}\label{equation:RLNid2}\\
&=2m(x,q,-1)-m(x,q,\sqrt{-q/x}) \notag \\
& =m(-qx^2,q^4,-q^{-1})-q^{-1}xm(-q^{-1}x^2,q^4,-q)\notag\\
\sum_{n= 0}^{\infty}{}^{\ast}\frac{(-1)^n(q;q^2)_n}{(-x)_{n+1}(-q/x)_n}&=m(x,q,-1)\label{equation:RLNid3}\\
\Big ( 1+\frac{1}{x}\Big )\sum_{n=0}^{\infty}\frac{(-1)^n(q;q^2)_nq^{(n+1)^2}}{(-xq,-q/x;q^2)_{n+1}}&=m(x,q,-1)-\frac{J_{1,2}^2}{2j(-x;q)}\label{equation:RLNid4}\\
\sum_{n= 0}^{\infty}\frac{(-1)^nq^{2n^2}(q^2;q^4)_{n}}{(-x;q^4)_{n+1}(-q^4/x;q^4)_{n}}&=m(x,q^2,q)+\frac{\overline{J}_{1,4}^2j(-xq^2;q^4)}{j(-x;q^4)j(xq;q^2)}\label{equation:RLNid5}
\end{align}}%
where \cite[Entry $12.3.3$]{ABI}
\begin{equation}
\sum_{n= 0}^{\infty}{}^{\ast}\frac{(-1)^n(q;q^2)_n}{(-x)_{n+1}(-q/x)_n}
:=\frac{1}{\overline{J}_{0,1}}\sum_{n=-\infty}^{\infty}\frac{(1+1/x)q^{n(n+1)/2}}{(1+xq^n)(1+q^n/x)}.\label{equation:sumstar-def}
\end{equation}
\end{proposition}
Although there does not appear to be a way to represent $m(x,q,z)$ as an Eulerian form, one can write $m(x,q,z)$ as a bilateral sum:
\begin{proposition}{\cite[p. $15$]{RLN}, \cite[Entry $3.4.7$]{ABII}}\label{proposition:bilateral-mxqz-prop} For $a,b\ne 0$,
\begin{align}
\sum_{n=0}^{\infty}\frac{a^{-n-1}b^{-n}}{(-1/a;q)_{n+1}(-q/b;q)_n}q^{n^2}&+\sum_{n=1}^{\infty}(-aq;q)_{n-1}(-b;q)_{n}q^{n}\label{equation:3.4.7-ABII}\\
&=\frac{(-aq)_{\infty}}{b(q)_{\infty}(-q/b)_{\infty}}j(-b;q)m\big ( a/b,q,-b\big ).\notag
\end{align}
\end{proposition}

We recall the universal mock theta function
\begin{equation}
g_3(x,q):=x^{-1}\Big ( -1 +\sum_{n=0}^{\infty}\frac{q^{n^2}}{(x)_{n+1}(q/x)_{n}} \Big ),\label{equation:g-def}
\end{equation}
as well as the easily shown
\begin{equation}
g_3(x,q)=\sum_{n=0}^{\infty}\frac{q^{n(n+1)}}{(x)_{n+1}(q/x)_{n+1}}.\label{equation:newgid}
\end{equation}

\section{Discussion of Zudilin's Proof} \label{section:zudilin}

We put Zudilin's proof \cite{Zu} into the context of \cite{Mo1}.  First, we recall some definitions
\begin{subequations}
\begin{equation*}
R(\omega;q):=\sum_{n=0}^{\infty}\frac{q^{n^2}}{(\omega q,q/\omega;q)_n},
\end{equation*}
\begin{equation*}
U(\omega;q):=\sum_{n=0}^{\infty}q^{n+1}(\omega q,q/\omega;q)_n.
\end{equation*}
\noindent For the interested reader we point out that
\begin{equation*}
U(\omega;q)=\sum_{n=0}^{\infty}\sum_{m=-\infty}^{\infty}u(m,n)(-\omega)^mq^n:=\sum_{n=0}^{\infty}(\omega q;q)_n(q\omega;q)_nq^{n+1},
\end{equation*}
where $u(m,n)$ is the number of strongly unimodal sequences of size $n$ with rank $m$ \cite[$(1.6)$]{FOR2}, \cite{BOPR}.  
\end{subequations}
By Proposition \ref{proposition:bilateral-mxqz-prop}:
\begin{equation}
R(\omega;q)+\Big ( 1-\omega \Big )\Big ( 1-\omega^{-1}\Big )U(\omega;q)=\Big ( 1-\omega \Big )\frac{j(\omega;q)}{J_1}m(\omega^2,q,\omega^{-1}).\label{equation:ram-bilat-id1}
\end{equation}
Set $\omega=-1$ in (\ref{equation:ram-bilat-id1}).  For $2k$, where $k$ is odd, the result follows from identity (\ref{equation:RLNid2}) of Proposition \ref{proposition:eulerian-mxqz-prop} with $x=1$.  Similarly, for $k$ even the result follows from  identity (\ref{equation:RLNid4}).  We demonstrate the former.  Using (\ref{equation:RLNid2}) with $x=1$ to eliminate $m(1,q,-1)$ from (\ref{equation:ram-bilat-id1}) gives
\begin{align*}
R(-1;q)+4U(-1;q)&=2 \frac{j(-1;q)}{J_1}m(1,q,-1)\\
&= \frac{j(-1;q)}{J_1}\sum_{n=0}^{\infty}\frac{(-1)^nq^{n^2}(q;q^2)_n}{(-q^2;q^2)_n}-\frac{J_{1,2}^2}{J_1}.
\end{align*}
Rearranging terms yields
\begin{align}
R(-1;q)+\frac{J_{1,2}^2}{J_1}=-4U(-1;q)+\frac{j(-1;q)}{J_1}\sum_{n=0}^{\infty}\frac{(-1)^nq^{n^2}(q;q^2)_n}{(-q^2;q^2)_n}.\label{equation:id1-pre}
\end{align}
As $q\rightarrow \zeta$ a root of unity of order $2k$, $k$ odd,  we see that the factor $j(-1;q)/J_1=2(-q;q)_{\infty}^2$ vanishes and that the sum in the second summand on the right-hand side of (\ref{equation:id1-pre}) has a finite value, so the second summand vanishes.  Thus for $\zeta$ a root of unity of order $2k$, $k$ odd, we have
\begin{align}
\lim_{q\rightarrow\zeta}\Big (R(-1;q)+\frac{J_{1,2}^2}{J_1}\Big )=-4U(-1;\zeta),\label{equation:id1-fin}
\end{align}
where the sum on the right-hand side of (\ref{equation:id1-fin}) terminates.  

We point out to the reader that
\begin{equation*}
\frac{1}{1-\omega} R(\omega;q)=1+\omega g_3( \omega,q),
\end{equation*}
and by \cite[Theorem $4.1$]{Mo1} we also have
\begin{align*}
\Big ( 1-\omega^{-1}\Big )U(\omega;q)&=-1-\omega g_3(\omega,q)+\frac{j(\omega;q)}{J_1}m(\omega^2,q,\omega^{-1}),\\
&=-1-\omega g_3(\omega,q)+\frac{j(\omega;q)}{J_1}m(\omega^2,q,-1) +\frac{1}{2}\frac{j(-\omega;q)^3j(q\omega^2;q^2)}{J_2^2j(\omega^4;q^2)}.
\end{align*}

\section{Proof of Theorem \ref{theorem:radial-new}}\label{section:new-theorem}
We introduce a few terms:
\begin{subequations}
\begin{equation*}
R_2(\omega;q):=\sum_{n=0}^{\infty}\frac{q^{2n^2+2n+1}}{(\omega q,q/\omega;q^2)_{n+1}},
\end{equation*}
\begin{equation*}
U_2(\omega;q):=\sum_{n=0}^{\infty}q^{2n+1}(\omega q,q/\omega;q^2)_{n}.
\end{equation*}
\end{subequations}
By Proposition \ref{proposition:bilateral-mxqz-prop}:
\begin{equation}
R_2(\omega;q)+U_2(\omega;q)=-\omega \cdot \frac{j(\omega q;q^2)}{J_2}m(\omega^2,q^2,\omega^{-1}q).\label{equation:rankcrank2}
\end{equation}
We have that
\begin{equation*}
R_2(\omega;q)=qg_3(\omega q,q^2)
\end{equation*}
and by \cite[Theorem $4.2$]{Mo1}:
\begin{align*}
U_2(\omega;q)&=-qg_3(\omega q,q^2)-\omega \cdot \frac{j(\omega q;q^2)}{J_2}m(\omega^2,q^2,-1)+\frac{1}{2}\frac{\omega j(-\omega q;q^2)^3j(\omega^2;q^4)}{J_4^2j(\omega^4;q^4)}\\
&=-qg_3(\omega q,q^2)-\omega \cdot \frac{j(\omega q;q^2)}{J_2}m(\omega^2,q^2,\omega^{-1}q).
\end{align*}
Setting $\omega=1$ in (\ref{equation:rankcrank2}) and using (\ref{equation:RLNid5}) with $x=1$ gives
\begin{align*}
qw(q)+U_2(1;q)&=-\frac{J_{1,2}}{J_2}m(1,q^2,q)\\
&=-\frac{J_{1,2}}{2J_2}\sum_{n=0}^{\infty}\frac{(-1)^nq^{2n^2}(q^2;q^4)_n}{(-q^4;q^4)_n^2}+\frac{\overline{J}_{1,4}^2\overline{J}_{2,4}}{\overline{J}_{0,4}J_{2}}.
\end{align*}
Rearranging terms yields
\begin{align}
qw(q)-\frac{\overline{J}_{1,4}^2\overline{J}_{2,4}}{\overline{J}_{0,4}J_{2}}=-U_2(1;q)-\frac{J_{1,2}}{2J_2}\sum_{n=0}^{\infty}\frac{(-1)^nq^{2n^2}(q^2;q^4)_n}{(-q^4;q^4)_n^2}.\label{equation:id2-pre}
\end{align}
Noting that $J_{1,2}/J_2=(q;q^2)_{\infty}^2$ we see for $q\rightarrow \zeta$ an odd root of unity that the second summand on the right-hand side of (\ref{equation:id2-pre}) vanishes and that the first summand on right-hand side of (\ref{equation:id2-pre}) terminates.  The result follows.

\section{More Radial Limit Theorems}\label{section:radials}
In this section we recall results from \cite{Mo1} in order to give four more examples of radial limits for mock theta functions.  Showing that the respective Eulerian forms have finite values for the given roots of unity $\zeta$ as $q\rightarrow \zeta$ is fairly straightforward, so we omit the arguments.  For examples, one could see \cite{BKLR}.  We point out to the reader that the four functions $V_k(\omega;q)$ are new results of \cite{Mo1}.

\subsection{Example 1} We define
\begin{equation}
S_1(\omega;q):=\Big ( 1+\omega^{-1}\Big )\sum_{n=0}^{\infty}\frac{(q;q^2)_n(-1)^nq^{(n+1)^2}}{(-\omega q,-q/\omega;q^2)_{n+1}},
\end{equation}
\begin{equation}
V_1(\omega;q):=\Big ( 1+\omega^{-1}\Big )\sum_{n=0}^{\infty}\frac{q^{2n+1}(-\omega q,-q/\omega ;q^2)_{n}}{(q;q^2)_{n+1}}.
\end{equation}
By  identity (\ref{equation:RLNid4}) of Proposition \ref{proposition:eulerian-mxqz-prop}:
\begin{equation}
S_1(\omega;q)=m(\omega,q,-1)-\frac{J_{1,2}^2}{2j(-\omega;q)}.\label{equation:ABII-6.3.6-dual}
\end{equation}
By \cite[$(4.20)$]{Mo1}:
\begin{align}
V_1(\omega;q)=-m(\omega,q,-1)+\frac{j(-\omega q;q^2)}{J_1}\Big ( 1-\omega g_3(-\omega,q^2)\Big )-\frac{1}{2}\frac{J_{1,2}^2}{j(-\omega;q)},
\end{align}
thus
\begin{equation}
S_1(\omega;q)+V_1(\omega;q)=\frac{j(-\omega q;q^2)}{J_1}\Big ( 1-\omega g_3(-\omega,q^2)\Big )-\frac{J_{1,2}^2}{j(-\omega;q)}.
\end{equation}
Noting that \cite[$(4.4)$]{HM}
\begin{equation}
1-g_3(-1,q^2)=\frac{1}{2}\sum_{n=0}^{\infty}\frac{q^{2n^2}}{(-q^2;q^2)_n^2},
\end{equation}
and that $\overline{J}_{1,2}/J_1=(-q;q^2)_{\infty}^2(-q;q)_{\infty}$ it follows that
\begin{theorem} If $\zeta$ is a primitive even order $2k$ root of unity, $k$ odd, then, as $q$ approaches $\zeta$ radially within the unit disk, we have that
\begin{equation}
\lim_{q\rightarrow \zeta} \Big ( S_1(1;q)+\frac{J_{1,2}^2}{\overline{J}_{0,1}}\Big )=-2\sum_{n=0}^{\frac{k-1}{2}}\frac{\zeta^{2n+1}(-\zeta;\zeta^2)_n^2}{(\zeta;\zeta^2)_{n+1}}.
\end{equation}
\end{theorem}


\subsection{Example 2} We define
\begin{equation}
S_2(\omega;q):=(1+\omega^{-1})\sum_{n= 0}^{\infty}\frac{q^{n+1}(-q)_{2n}}{(q\omega,q/\omega;q^2)_{n+1}},
\end{equation}
\begin{equation}
V_2(\omega;q):=\Big ( 1+\omega^{-1}\Big )\sum_{n=0}^{\infty}\frac{(\omega q,q/\omega;q^2)_nq^{2n+1}}{(-q)_{2n+1}}.
\end{equation}
By identity (\ref{equation:RLNid1}) of  Proposition \ref{proposition:eulerian-mxqz-prop}:
\begin{equation}
S_2(\omega;q)=-m(\omega,q^2,q).
\end{equation}
By \cite[Theorem $4.3$]{Mo1}:
\begin{align}
V_2(\omega;q)=2m(\omega,q^2,-1)-\frac{j(\omega q;q^2)}{\overline{J}_{1,4}}m(\omega,q,-1)
 -\frac{J_{1,2}^2}{\overline{J}_{0,1}}\frac{j(\omega q;q^2)}{j(-\omega;q)}.
\end{align}
Using (\ref{equation:m-change-z}) gives
\begin{align}
2S_2(\omega;q)+V_2(\omega;q)&=-\frac{j(\omega q;q^2)}{\overline{J}_{1,4}}m(\omega,q,-1)
+\frac{2J_2^3\overline{J}_{1,2}j(-q\omega;q^2)}{J_{1,2}\overline{J}_{0,2}j(q\omega;q^2)j(-\omega;q^2)}\label{equation:id4-pre}\\
&\ \ \ \ \ -\frac{J_{1,2}^2}{\overline{J}_{0,1}}\frac{j(\omega q;q^2)}{j(-\omega;q)}.\notag
\end{align}
Setting $\omega=1$ in (\ref{equation:id4-pre}), using (\ref{equation:RLNid2}) with $x=1$, and noting $J_{1,2}/\overline{J}_{1,4}=(q;q^2)_{\infty}^3$ gives
\begin{theorem} If $\zeta$ is a primitive order $2k+1$ root of unity, then, as $q$ approaches $\zeta$ radially within the unit disk, we have that
\begin{align}
\lim_{q\rightarrow \zeta}\Big ( 2S_2(1;q)-\frac{\overline{J}_{1,2}^3}{2J_{1,4}^2}\Big)=-2\sum_{n=0}^{k}\frac{(\zeta,\zeta^2)_n^2\zeta^{2n+1}}{(-\zeta;\zeta)_{2n+1}}.
\end{align}
\end{theorem}

\subsection{Example 3} We define
\begin{equation}
S_3(\omega;q):=\sum_{n=0}^{\infty}\frac{(q;q^2)_{n}(-1)^nq^{n^2}}{(-\omega;q^2)_{n+1}(-q^2/\omega;q^2)_{n}},
\end{equation}
\begin{equation}
V_3(\omega;q):=\Big ( 1+\omega^{-1}\Big)\sum_{n=0}^{\infty}\frac{q^{2n+2}(-\omega q^2,-q^2/\omega;q^2)_{n}}{(q;q^2)_{n+1}}.
\end{equation}
By identity (\ref{equation:RLNid2}) of  Proposition \ref{proposition:eulerian-mxqz-prop}:
\begin{equation}
S_3(\omega;q)= m(\omega,q,-1)+\frac{J_{1,2}^2}{2j(-\omega;q)}.\label{equation:6.3.9-dual-final}
\end{equation}
By \cite[$(4.29)$]{Mo1}:
\begin{equation}
V_3(\omega;q)=-m(\omega,q,-1)+\frac{j(-\omega;q^2)}{J_1}\frac{q}{\omega}g_3(-\omega q,q^2)+\frac{1}{2}\frac{J_{1,2}^2}{j(-\omega;q)}.
\end{equation}
Thus
\begin{equation}
S_3(\omega;q)+V_3(\omega;q)=\frac{j(-\omega;q^2)}{J_1}\frac{q}{\omega}g_3(-\omega q,q^2)+\frac{J_{1,2}^2}{j(-\omega;q)}.
\end{equation}
Noting that \cite[$(4.8)$]{HM}
\begin{equation}
g_3(-q,q^2)=\sum_{n=0}^{\infty}\frac{q^{2n(n+1)}}{(-q;q^2)_{n+1}^2},
\end{equation}
and $\overline{J}_{0,2}/J_1=2(-q^2;q^2)_{\infty}^2(-q;q)_{\infty}$ it follows that
\begin{theorem} If $\zeta$ is a primitive even order $4k$ root of unity, then, as $q$ approaches $\zeta$ radially within the unit disk, we have that
\begin{equation}
\lim_{q\rightarrow \zeta} \Big ( S_3(1;q)-\frac{J_{1,2}^2}{\overline{J}_{0,1}}\Big )=-2\sum_{n=0}^{k-1}\frac{\zeta^{2n+2}(-\zeta^2;\zeta^2)_n^2}{(\zeta;\zeta^2)_{n+1}}.
\end{equation}
\end{theorem}

\subsection{Example 4}
We define
\begin{equation}
S_4(\omega;q):=\sum_{n= 0}^{\infty}{}^*\frac{(-1)^n(q;q^2)_n}{(-\omega;q)_{n+1}(-q/\omega;q)_n},
\end{equation}
\begin{equation}
V_4(\omega;q):=\Big ( 1+\frac{1}{\omega}\Big )\sum_{n=0}^{\infty}\frac{(-\omega q,-q/\omega;q)_nq^{n+1}}{(q;q^2)_{n+1}}.
\end{equation}
By identity (\ref{equation:RLNid3}) of  Proposition \ref{proposition:eulerian-mxqz-prop}:
\begin{equation}
S_4(\omega;q)=m(\omega,q,-1).
\end{equation}
By \cite[Theorem $4.4$]{Mo1}:
\begin{equation}
V_4(\omega;q)=-m(\omega,q,-1)+\frac{j(-\omega;q)}{J_{1,2}}m(\omega^2,q^2,-1) -\omega \frac{J_4^3}{J_2^3}\frac{j(\omega ;q)j(q\omega^2;q^2)}{j(\omega^4;q^4)}.
\end{equation}
Thus
\begin{equation}
S_4(\omega;q)+V_4(\omega;q)=\frac{j(-\omega;q)}{J_{1,2}}m(\omega^2,q^2,-1) -\omega \frac{J_1^2J_4^2}{J_2^2}\frac{j(q\omega^2;q^2)}{j(-\omega;q)j(-\omega^2;q^2)}.
\end{equation}
Using (\ref{equation:RLNid1}) with $x=1$ and (\ref{equation:m-change-z}) it is straightforward to show
\begin{theorem} If $\zeta$ is a primitive even order $2k$ root of unity, then, as $q$ approaches $\zeta$ radially within the unit disk, we have that
\begin{equation}
\lim_{q\rightarrow \zeta} \Big ( S_4(1;q)+\frac{J_{1}^3}{\overline{J}_{0,1}^2}
-\frac{\overline{J}_{1,2}^4}{2J_1^3}\Big )=-2\sum_{n=0}^{k-1}\frac{\zeta^{n+1}(-\zeta;\zeta)_n^2}{(\zeta;\zeta^2)_{n+1}}.
\end{equation}
\end{theorem}
\section{Gordon and McIntosh's universal mock theta function}\label{section:GM}

We recall a universal mock theta function of Gordon and McIntosh \cite[$(4.11)$]{GM}:
\begin{equation}
g_2(x,q):=\sum_{n=0}^{\infty}\frac{q^{\tfrac{1}{2}n(n+1)}(-q;q)_n}{(x;q)_{n+1}(q/x)_{n+1}},
\end{equation}
with \cite[$(4.7)$]{GM} and \cite[$(4.10)$]{GM}:
\begin{equation}
g_2(x;q)=\frac{J_2^3}{J_{1,2}j(x^2;q^2)}-xq^{-1}m(x^2q^{-1},q^2,q).\label{equation:g2-mxqz}
\end{equation}
We follow the methods of \cite{Mo1}.  The tail of the bilateral series can be written
\begin{equation}
\sum_{n=0}^{\infty}\frac{1}{2}\frac{q^n(q/x;q)_n(x;q)_n}{(-q;q)_n}:=f(x).\label{equation:g2tail}
\end{equation}
If we make the ansatz
\begin{equation}
f(x)=xq^{-1}m(x^2q^{-1},q^2,q)
\end{equation}
such that the $m(x,q,z)$ expressions cancel upon addition of (\ref{equation:g2-mxqz}) and (\ref{equation:g2tail}), we would expect the functional equation to be
\begin{equation}
f(qx)+x^2f(x)-x=0,
\end{equation}
but instead we find
\begin{equation}
f(qx)+x^2f(x)-x=(1-x)\frac{j(x;q)}{J_2}.
\end{equation}
We solve for $f(x)$:
\begin{equation}
f(x)=x^{-1}+(x^{-2}-x^{-1})\frac{j(x;q)}{J_2}-x^{-2}f(qx).\label{equation:f-fnq}
\end{equation}
Iterating (\ref{equation:f-fnq}) and using the heuristic \cite[$(2.5)$]{HM}, suggests
\begin{equation}
f(x)=x^{-1}m(x^{-2}q,q^2,z)-\frac{j(x;q)}{2J_2}g_2(-x;q)+theta,
\end{equation}
where $z$ is arbitrary.  Some numerical work quickly leads to
\begin{equation}
f(x)=xq^{-1}m(x^{2}q^{-1},q^2,-1)-\frac{j(x;q)}{2J_2}g_2(-x;q)-\frac{1}{2x}\frac{J_{2,4}^2j(x;q)}{j(-x;q)j(-qx^2;q^2)}.\label{equation:gm-mixed}
\end{equation}
\begin{theorem} \label{theorem:gm-mixed} Identity (\ref{equation:gm-mixed}) is true.
\end{theorem}
\noindent Adding (\ref{equation:g2-mxqz}) and (\ref{equation:gm-mixed}) and using (\ref{equation:m-change-z}), our mixed mock modular bilateral $q$-series for $g_2(x,q)$ is thus
\begin{corollary} \label{corollary:master} We have
\begin{align}
\sum_{n=0}^{\infty}&\frac{q^{\tfrac{1}{2}n(n+1)}(-q;q)_n}{(x;q)_{n+1}(q/x)_{n+1}}+\sum_{n=0}^{\infty}\frac{1}{2}\frac{q^n(q/x;q)_n(x;q)_n}{(-q;q)_n}=-\frac{j(x;q)}{2J_2}g_2(-x;q)\\
&+\frac{J_2^3}{J_{1,2}j(x^2;q^2)}+\frac{1}{2x}\frac{J_{2}^{10}j(-x^2;q^2)}{J_{1}^4J_{4}^4j(x^2;q^2)j(-qx^2;q^2)}-\frac{1}{2x}\frac{J_{2,4}^2j(x;q)}{j(-x;q)j(-qx^2;q^2)}.\notag
\end{align}
\end{corollary}
\begin{proof}[Proof of Theorem \ref{theorem:gm-mixed}:]   The proof is straightforward but lengthy, so we give a quick sketch.  Using Bailey pair techniques, see for example Lovejoy \cite{L}
\begin{equation}
2J_2f(x)=f_{1,3,6}(xq,q^5,q)+qf_{1,3,6}(xq^2,q^7,q).
\end{equation}
Using the $f_{1,3,6}(x,y,q)$ specialisation of the main result of \cite{Mo2} and doing some calculating similar to the many examples in \cite{HM} yields the result.  
\end{proof}
Corollary \ref{corollary:master} allows us to compute the radial limits of second, sixth, eighth and tenth order mock theta functions that can be expressed in terms of $g_2(x,q)$, e.g. \cite[$(5.2)$, $(5.10)$, $(5.18)$]{GM}.    For examples, we recall the second order $B_2(q)$ from \cite[$(5.2)$]{GM}
\begin{equation}
B_2(q):=\sum_{n\ge 0}\frac{q^{n}(-q;q^2)_n}{(q;q^2)_{n+1}}=\sum_{n\ge 0}\frac{q^{n^2+n}(-q^2;q^2)_n}{(q;q^2)_{n+1}^2}=g_2(q,q^2),
\end{equation}
as well as two of the tenth orders \cite[$(5.18)$]{GM}
\begin{align}
{\phi}_{10}(q)&:=\sum_{n\ge 0}\frac{q^{\binom{n+1}{2}}}{(q;q^2)_{n+1}}=\frac{J_{10}^2\overline{J}_{2,5}}{J_5J_{2,10}}+2qg_2(q^2,q^5),\\
\psi_{10}(q)&:=\sum_{n\ge 0}\frac{q^{\binom{n+2}{2}}}{(q;q^2)_{n+1}}=-\frac{qJ_{10}^2\overline{J}_{1,5}}{J_5J_{4,10}}+2qg_2(q,q^5).
\end{align}
We point out that there are more such expansions in \cite[$(5.2)$]{GM} and \cite[$(5.18)$]{GM}.

As in previous sections, note that $j(x;q)/J_2=(x;q)_{\infty}(q/x;q)_{\infty}(q;q^2)_{\infty}$.  If we let $q\rightarrow q^2$ and $x\rightarrow q$ in Corollary \ref{corollary:master} then we have a radial limit result for the second order mock theta function $B_2(q):$  
\begin{corollary} If $\zeta$ is a primitive odd order $2k+1$ root of unity, then, as $q$ approaches $\zeta$ radially within the unit disk, we have that
\begin{equation}
\lim_{q\rightarrow \zeta} \Big ( B_2(q)-\frac{J_4^5}{J_2^4}-\frac{1}{4}q^{-1}\frac{J_4^{17}}{J_2^8J_8^8}+\frac{1}{4}q^{-1}\frac{J_4J_1^4}{J_2^2J_8^2}\Big )=-\frac{1}{2}\sum_{n=0}^{k}\frac{\zeta^{2n}(\zeta;\zeta^2)_n^2}{(-\zeta^2;\zeta^2)_{n}}.
\end{equation}
\end{corollary}

\begin{remark}Identity (\ref{equation:gm-mixed}) is reminiscent of \cite[$(2.29)$]{L}.
\end{remark}

We can also easily form a radial limit result for the tenth order $\psi_{10}(q):$
\begin{corollary} If $\zeta$ is a primitive $10k+\ell$, $\ell\in\{1,3,5,7,9\}$, root of unity, then, as $q$ approaches $\zeta$ radially within the unit disk, we have that
\begin{equation}
\lim_{q\rightarrow \zeta} \Big ( \psi_{10}(q)+\frac{qJ_{10}^2\overline{J}_{1,5}}{J_5J_{4,10}}
-\frac{2qJ_{10}^4}{J_5^2J_{2,10}}-\frac{J_{10}^{10}\overline{J}_{2,10}}{J_{5}^4J_{20}^4J_{2,10}\overline{J}_{7,10}}
+\frac{J_{10,20}^2J_{1,5}}{\overline{J}_{1,5}\overline{J}_{7,10}}\Big )=-\zeta \cdot F(\ell).
\end{equation}
where
\begin{equation}
F(\ell)=\begin{cases}
\sum_{n=0}^{2k}\frac{\zeta^{5n}(\zeta;\zeta^5)_n(\zeta^4;\zeta^5)_n}{(-\zeta^5;\zeta^5)_{n}} & \mbox{if } \ell =1,\\ \\
\sum_{n=0}^{4k+1}\frac{\zeta^{5n}(\zeta;\zeta^2)_n(\zeta^4;\zeta^5)_n}{(-\zeta^5;\zeta^5)_{n}} & \mbox{if } \ell = 3,\\ \\
\Big ( 1-\frac{(\zeta;\zeta^5)_{2k+1}(\zeta^4;\zeta^5)_{2k+1}}{(-\zeta^5;\zeta^5)_{2k+1}}\Big )^{-1}
\sum_{n=0}^{2k}\frac{\zeta^{5n}(\zeta;\zeta^5)_n(\zeta^4;\zeta^5)_n}{(-\zeta^5;\zeta^5)_{n}} & \mbox{if } \ell =5,\\ \\
\sum_{n=0}^{4k+2}\frac{\zeta^{5n}(\zeta;\zeta^5)_n(\zeta^4;\zeta^5)_n}{(-\zeta^5;\zeta^5)_{n}} & \mbox{if } \ell =7,\\ \\
\sum_{n=0}^{2k+1}\frac{\zeta^{5n}(\zeta;\zeta^5)_n(\zeta^4;\zeta^5)_n}{(-\zeta^5;\zeta^5)_{n}} & \mbox{if } \ell = 9.
\end{cases}
\end{equation}

\end{corollary}

We can also produce a similar result for the tenth order $\phi_{10}(q)$.
\section*{Acknowledgements}


\end{document}